\documentclass[12pt]{article}
\usepackage{amsmath, amsthm, amsfonts, amssymb}
\usepackage{color}
\newtheorem{theorem}{Theorem}[section]
\newtheorem{corollary}[theorem]{Corollary}

\newtheorem{proposition}[theorem]{Proposition}

\newtheorem{remarks}[theorem]{Remarks}
\newtheorem{remark}[theorem]{Remark}

\numberwithin{equation}{section}

\title{\bf Revisiting Li-Yau type inequalities on Riemannian manifolds}
\author{Bin Qian\thanks{Department of Mathematics and Statistics, Changshu Institute of Technology, Changshu, Jiangsu 215500, China.\newline
E-mail: binqiancs@yahoo.com, binqiancn@gmail.com}
 }\date{}

\def\<{\langle}
\def\>{\rangle}

\begin{document}
\maketitle \begin{abstract}
Inspired Yau's work (Comm. Anal. Geom., 1994), in this short note we provide a new version  of Li-Yau gradient estimate for the linear heat equation, which generalizes some known results and gives new gradient estimates. Also we explain the different known results as different cases  here.
\end{abstract}

 \vskip10pt \textbf{2020
MR Subject Classification:} 53C21 35K05

\section{Introduction}

In their well known work, Li and Yau \cite{LY} proved an upper bound on the gradient estimates of positive solutions to the heat equation, which is socalled Li-Yau inequality. It gives parabolic Harnack inequality, which provides a comparison between heat at two different points in space and at different times. It also gives  good bounds of the associated heat kernel, Green functions and lower bound of Dirichlet or Neumann eigenvalue.  Since then it becomes a powerful tool in heat kernel analysis, PDE, entropy theory, differential geometry etc. It also plays an important role in the Perelman's solution to the Poincar\'e conjecture. More precisely,  it claims that, for an $n$-dimensional compact Riemannian manifold with Ricci curvature bounded below by $-K(K\ge0)$, if $u$ is a positive solution to the heat equation $\partial_t u=\Delta u$, then for all $\alpha>1$,
\begin{equation}\label{LY-grad}
|\nabla \log u|^2 -\alpha (\log u)_t\le \frac{n\alpha^2}{2t}+ \frac{n\alpha^2K}{2(\alpha-1)} .
\end{equation}

If we take the Gaussian kernel $u(x,t)=\frac{1}{(4\pi t)^{\frac{n}{2}}}e^{-\frac{|x|^2}{4t}}$ on $\mathbb{R}^n$, in this case the inequality \eqref{LY-grad} is equality for $K=0, \alpha=1$. There are a series work on improving this inequality for negative curvature and  for small time and large time, see  \cite{Hamilton93,SZ06,Li-Xu11,BG11,BBG17,Yu-Zhao20} and the references therein. We briefly recall them as follows:

 B. Davies in \cite{Davies89} improved the  estimate to
\begin{equation}\label{LY-Hamil}
\frac{|\nabla u|^2}{u^2}-\alpha \frac{u_t}{u}\le\frac{n\alpha^2}{2t}
 +\frac{n\alpha^2K}{4(\alpha-1)}
\end{equation}
holds for any $\alpha>1, t>0$.

In \cite{Hamilton93}, R. S. Hamilton proved a new gradient estimate for the
heat equation, the new viewpoint is that we can see constant  $\alpha$ in \cite{LY} as function  of time $t$:  \begin{align}\label{Hamil}
\frac{|\nabla u|^2}{u^2}-e^{2Kt}\frac{u_t}{u}\le e^{4Kt}\frac{n}{2t}.
\end{align}

Bakry-Qian in \cite{BQ99}  improved the above inequality to
the following: for all $t>0$,
\begin{equation}\label{Yau2}
\frac{|\nabla u|^2}{u^2}-\frac{u_t}{u}\le \sqrt{nK}\sqrt{\frac{|\nabla
u|^2}{u^2}+\frac{n}{2t}+\frac{nK}{4} }+\frac{n}{2t}.
\end{equation}
and
\begin{equation}\label{LX-eq1}
|\nabla f|^2-\left(1+\frac23Kt\right)f_t\le
\frac{n}{2t}+\frac{nK}{2}\left(1+\frac13Kt\right).
\end{equation}

Li-Xu in \cite{Li-Xu11} re-found \eqref{LX-eq1} and further got the following gradient estimate:
\begin{equation}\label{LX-eq2}
|\nabla
f|^2-\left(1+\frac{\sinh(Kt)\cosh(Kt)-Kt}{\sinh^2(Kt)}\right)f_t\le
\frac{nK}{2}(\coth(Kt)+1).
\end{equation}

Meanwhile, Bakry and Ledoux in \cite{Bakry-Ledoux06} found a logarithmic Sobolev form of the
Li-Yau parabolic inequality by semigroup method , which  generalizes \eqref{LY-grad}. See also \cite{BG11} and \cite{BV12} in this direction.  The author himself in \cite{Qian14} provides  a general form of gradient estimate which generalizes \eqref{LX-eq1} and \eqref{LX-eq2}, see also \cite{BG11,Yu-Zhao20}.

 In the very recent, Bakry, Bolley and Gentil in \cite{BBG17} have obtained the following  refined global Li-Yau inequality:
\begin{equation}\label{BBG-1}
\frac{|\nabla u|^2}{u^2}<\frac{n}{2}\Phi_t\left(\frac{4}{-nK}\frac{\Delta u}{u}\right),
\end{equation}
where
$$
\Phi_t(x)=\begin{cases}
\frac{-K}{2}\left(x-2+2\sqrt{1-x}\coth(-Kt\sqrt{1-x})\right), & x\le 1\\
 \frac{-K}{2}\left(x-2+2\sqrt{x-1}\cot(-Kt\sqrt{x-1})\right), &1\le x<1+\frac{\pi^2}{K^2t^2}.
\end{cases}
$$
This bound  works  for both negative and positive curvature.   They show estimate \eqref{BBG-1} is stronger than \eqref{LY-grad}, \eqref{LY-Hamil}, \eqref{Yau2}, \eqref{LX-eq1} and \eqref{LX-eq2},  see  section 5 in \cite{BBG17}.

Inspired Yau's work \cite{Yau94}, see also \cite{MZS08}, we give  a new version  of gradient estimate for the linear heat equation, if the parameter functions satisfy certain nonlinear condition, see Theorem \ref{YHM-thm0}. We further analysis the nonlinear condition, and divide it into two cases (see Theorem \ref{thm1}). Then we show  the known results \eqref{LY-grad},  \eqref{LX-eq1} and \eqref{LX-eq2} are particular ones of  the respective cases (see Corollary \ref{cor1} and Theorem \ref{thm-case2}), also we refind the Hamilton type gradient estimate which refines \eqref{Hamil}, see Corollary \ref{cor2} and Corollary \ref{cor2-2}. Evenmore, we can derive new estimate refining \eqref{LY-Hamil} by combining the two cases,  see Remark \ref{rem2}.

\section{Main results}

\begin{theorem}\label{YHM-thm0}

Let $M$ be a compact manifold without
boundary, $Ricci(M)\ge-K$.  Let $u$ be any positive solution of the  heat equation \begin{equation}\label{heat} u_t=\Delta u
\end{equation} on $M$. For $c(t)\in C^1(0,\infty)$ and  $\alpha\in C^1[0,\infty)$ be two arbitrary  positive functions,  if the conditions below are satisfied: \begin{description}
\item{(A)} $c(0+):=\lim\limits_{t\downarrow0}c(t)=+\infty$;
\item{(B)} (nonlinear condition) the following quadratic inequality
holds for all $x\ge 0$,$t>0$,
$$-\frac{2(1-\alpha(t))^2}{n\alpha^2(t)}x^2+\left(2K+\frac{4(1-\alpha)c(t)}{n\alpha^2}
-\frac{\alpha'}{\alpha}\right)x-\frac{2c^2(t)}{n\alpha^2}-c'(t)+\frac{\alpha'(t)}{\alpha(t)}c(t)\le0.
$$
 \end{description}
Then we have for all $t>0$, \begin{equation}\label{YHM-LYH0} \left|\nabla
\log u\right|^2-\alpha(t)(\log u)_t-c(t)\le 0. \end{equation}

\end{theorem}
\begin{proof}
Without loss of any generalization, we can assume $u\ge \varepsilon>0$, otherwise $u_{\varepsilon}=u+\varepsilon$, obviously $u_{\varepsilon}$ is also a solution to the heat equation \eqref{heat}, then letting $\varepsilon\to0$ gives the desired result. For arbitrary  positive function $\alpha\in C^1[0,\infty)$, denote $$F=|\nabla \log u|^2-\alpha(t)(\log u)_t.$$
Direct computation gives
$$\aligned
\partial_tF&=2\nabla \log u\cdot\nabla \frac{u_t}{u}-\alpha'(t)(\log u)_t-\alpha(t)(\log u)_{tt}\\
&=2\nabla \log u\cdot\nabla\frac{\Delta u}{u}-\alpha'(t)\frac{\Delta u}{u}-\alpha(t)(\log u)_{tt}\\
&=2\nabla \log u\cdot\nabla (\Delta \log u+|\nabla \log u|^2)-\alpha' \Delta \log u-\alpha'|\nabla\log u|^2-\alpha(t)(\log u)_{tt},
\endaligned$$
$$
\Delta F=\Delta |\nabla \log u|^2-\alpha\Delta(\log u)_t,
$$
and
$$\aligned
\nabla \log u\cdot\nabla F&=\nabla \log u\cdot\nabla |\nabla \log u|^2-\alpha \nabla \log u\cdot\nabla (\log u)_t\\
&=\nabla \log u\cdot\nabla |\nabla \log u|^2-\frac12\alpha \partial_t( |\nabla \log u|^2).
\endaligned$$

This yields
\begin{align}
(\partial_t-\Delta)F&=-\left(\Delta |\nabla\log u|^2-2\nabla\log u\cdot\nabla \Delta\log u\right)+2\nabla \log u\cdot\nabla |\nabla \log u|^2-\alpha'\Delta \log u\nonumber\\
&\hskip12pt -\alpha'|\nabla \log u|^2+\alpha\left((\Delta-\partial_t)\log u\right)_t\nonumber\\
&=-\left(\Delta |\nabla\log u|^2-2\nabla\log u\cdot\nabla \Delta\log u\right)+2\nabla \log u\cdot\nabla |\nabla \log u|^2-\alpha'\Delta \log u\nonumber\\
&\hskip12pt -\alpha'|\nabla \log u|^2-\alpha \partial_t(|\nabla \log u|^2)\nonumber\\
&=2\nabla \log u\cdot\nabla F-2\left(|\mbox{Hess}\log u|^2+Ricci (\nabla \log u,\nabla \log u)\right)\nonumber\\
&\hskip12pt -\alpha'(t)\left(\Delta\log u+|\nabla \log u|^2\right).\label{eqdiff}
\end{align}
where the last equality follows from the Bochner-Weitzenb\"ock formula.

 For the functional $\Psi:=F-c(t)$, we have obviously $\Psi(0+)\le0$. Assuming that: if $\Psi=F-c(t)\le 0$ for $t\le
t_0$ and $\Psi(x_0,t_0)=F(x_0,t_0)-c(t_0)=0$ for some $x_0\in M$, by the maximum principle  we would have at the point $(x_0,t_0)$,
$$
\frac{d}{dt}\Psi\ge 0,\ \Delta \Psi\le 0,\ \nabla
\Psi=0.
$$
Noticing that at $(x_0,t_0)$, $$
\Delta \log u=-\frac{F}{\alpha}+\frac{1-\alpha}{\alpha}|\nabla \log u|^2=-\frac{c}{\alpha}+\frac{1-\alpha}{\alpha}|\nabla \log u|^2,
$$
substituting into \eqref{eqdiff}, it follows at $(x_0,t_0)$
\begin{align}
0\le (\partial_t-\Delta) \Psi&=2\nabla \log u\cdot\nabla \Psi-2\left(|\mbox{Hess}\log u|^2+Ricci (\nabla \log u,\nabla \log u)\right)\nonumber\\
&\hskip12pt -\alpha'(t)\left(\Delta\log u+|\nabla \log u|^2\right)-c'\nonumber\\
&\le -\frac{2}{n}(\Delta \log u)^2+2K|\nabla \log u|^2-\alpha'(t)\left(\Delta\log u+|\nabla \log u|^2\right)-c'\nonumber\\
&=-\frac{2}{n}\left(\frac{1-\alpha}{\alpha}|\nabla \log u|^2-\frac{c}{\alpha}\right)^2+2K|\nabla \log u|^2-\alpha'\left(\frac{|\nabla \log u|^2}{\alpha}-\frac{c}{\alpha}\right)-c'\nonumber\\
&=-\frac{2}{n}\left(\frac{\alpha-1}{\alpha}\right)^2|\nabla \log u|^4+\left(2K+\frac{4(1-\alpha)c}{n\alpha^2}-\frac{\alpha'}{\alpha}\right)|\nabla \log u|^2\nonumber\\
&\hskip 12pt -\frac{2c^2}{n\alpha^2}-c'+\frac{\alpha'c}{\alpha}.\nonumber
\end{align}
Combining with nonlinear condition (B), we have at $(x_0,t_0)$,
$$
\partial_t \Psi=0, \Delta \Psi=0, \nabla \Psi=0.
$$
 By the strong maximum principle, we see that $ \Psi(x_0,t)\le 0$, $\forall t\in (t_0,t_0+\delta)$ for some $\delta>0$, thus we have
 $$
\Psi =|\nabla \log u|^2-\alpha (\log u)_t -c(t)\le 0, \forall t>0.
 $$
  We complete the proof.  \end{proof}

\begin{proposition}\label{prop-basic}
For the quadratic function $f(x)=ax^2+bx+c$, where $a,b,c$ are some constants and $a<0$.  $f(x)\le 0$ holds for all $x\ge 0$ if and only if the following two cases satisfy any one of them: {\bf Case 1:} $b\le 0$ and $c\le 0$; {\bf Case 2:} $b^2-4ac\le 0$.

\end{proposition}
\begin{proof}  The proof is elementary.
\end{proof}

Combining with Theorem \ref{YHM-thm0} and Proposition \ref{prop-basic}, we have the following theorem.

\begin{theorem}\label{thm1}
Let $M$ be a compact manifold without
boundary, $Ricci(M)\ge-K$.  Let $u$ be any positive solution of the  heat equation $u_t=\Delta u
$ on $M$. For $c(t)\in C^1(0,\infty)$, and  $\alpha\in C^1[0,\infty)$ be an arbitrary  positive function, if the conditions below are satisfied: \begin{description}
\item{(a).} $c(0+):=\lim\limits_{t\downarrow0}c(t)=+\infty$;
\item{(b).} Suppose the following two cases satisfy any one of them:
 \begin{align}
&\mbox{ \bf{  Case\  1:}}\ \ \forall t>0,
(1-\alpha)c(t)\le \frac{n\alpha(\alpha'-2K\alpha)}{4}  \ \mbox{and} \ \left(\frac{\alpha}{c}\right)'\le \frac{2}{n\alpha}.\label{case1}\\
& \mbox{ \bf{ Case\ 2:} } \ \forall t>0,
(1-\alpha)^2 c'(t)\ge (1-\alpha)(2K-\alpha')c(t)+ \frac{n(2K\alpha-\alpha')^2}{8}.\label{case2}
\end{align}
 \end{description}

Then we have for all $t>0$,  \begin{equation}\label{YHM-LYH0} \left|\nabla
\log u\right|^2-\alpha(t)(\log u)_t-c(t)\le 0. \end{equation}

\end{theorem}
\begin{proof}
The proof follows from the combination of Theorem \ref{YHM-thm0} and Proposition  \ref{prop-basic}. By the assumption of $c(0+)=+\infty$, we have Condition A holds in Theorem \ref{YHM-thm0}. To verify the nonlinear condition $B$ in Theorem \ref{YHM-thm0} , by Proposition \ref{prop-basic}, we divide it  into  two cases.  \begin{description}
\item{(1).} Case 1 in Proposition \ref{prop-basic}:
$$\begin{cases}
2K+\frac{4(1-\alpha)c(t)}{n\alpha^2}
-\frac{\alpha'}{\alpha}&\le 0,\\
-\frac{2c^2(t)}{n\alpha^2}-c'(t)+\frac{\alpha'(t)}{\alpha(t)}c(t)&\le0.
\end{cases}
$$
We can  rewrite  them as  \eqref{case1}.

\item{(2).} Case 2 in Proposition \ref{prop-basic}:
$$
\left(2K+\frac{4(1-\alpha)c(t)}{n\alpha^2}
-\frac{\alpha'}{\alpha}\right)^2-\frac{8(1-\alpha(t))^2}{n\alpha^2(t)}\left(\frac{2c^2(t)}{n\alpha^2}+c'(t)-
\frac{\alpha'(t)}{\alpha(t)}c(t)\right)\le0,
$$
which can be rewritten as \eqref{case2}.
\end{description}
By Theorem \ref{YHM-thm0}, the proof is completed.
\end{proof}
\begin{remark}\label{rem1}
The conclusion of Theorem \ref{thm1} holds for all $t\le T$ ($T>0$) if the second assumption is replaced by \eqref{case1} holds for $t\in (0,t_0]$ and \eqref{case2} holds for $t\in (t_0,T]$ for some $0<t_0<T$. This is because condition B in Theorem \ref{YHM-thm0} still holds in this case.
\end{remark}

\subsection{Case 1 in Theorem \ref{thm1} }

%\begin{theorem}\label{thm-case1}
%Let $M$ be a compact manifold without
%boundary, $Ricci(M)\ge-K$.  Let $u$ be any positive solution of the  heat equation $u_t=\Delta u
%$ on $M$. For any arbitrary $C^1[0,\infty)$ function $\alpha> 1$, we have
%\begin{equation}\label{thm-case1-eq1}
%|\nabla \log u|^2-\alpha (\log u)_t\le c(t):=\max\left\{\frac{n\alpha(2K\alpha-\alpha')}{4(\alpha-1)}, \frac{n\alpha}{2\int_0^t \frac{1}{\alpha}ds}\right\}.
%\end{equation}

%\end{theorem}
%\begin{proof}  We only need to verify {\bf Case 1} in Theorem \ref{thm1}, i.e. \eqref{case1}. By $\left(\frac{1}{c}\right)'\le \frac{2}{n\alpha^2}$ and $c(0+)=\infty$, we can obtain $c(t)\ge \frac{n\alpha}{2\int_0^t \frac{1}{\alpha}ds}$. Meanwhile $c(t)\ge \frac{n\alpha(2K\alpha-\alpha')}{4(\alpha-1)}$, hence $c(t)\ge \max\left\{\frac{n\alpha(2K\alpha-\alpha')}{4(\alpha-1)}, \frac{n\alpha}{2\int_0^t \frac{1}{\alpha}ds}\right\}$. Although the function  $\max\left\{\frac{n\alpha(2K\alpha-\alpha')}{4(\alpha-1)}, \frac{n\alpha}{2\int_0^t \frac{1}{\alpha}ds}\right\}$ is not $C^1$ in time $t$, we can using the approximation method to complete the proof.
%\end{proof}

\begin{corollary}\label{cor1}
Let $M$ be a compact manifold without boundary, $Ricci(M)\ge-K (K\ge0)$.
Let $u$ be any positive solution of the  heat equation $u_t=\Delta u $ on $M$ and let  $\alpha>1$ be a positive constant. We have for all $t>0$,
\begin{equation}\label{cor1-eq1}
|\nabla \log u|^2-\alpha (\log u)_t\le \frac{n\alpha^2}{2}\max\left\{\frac{1}{t}, \frac{K}{\alpha-1}\right\}
\end{equation}
\end{corollary}
\begin{proof} We verify Case 1 in Theorem \ref{thm1}, i.e. \eqref{case1}. By $\left(\frac{1}{c}\right)'\le \frac{2}{n\alpha^2}$ and $c(0+)=\infty$, we can obtain $\frac{1}{c(t)}\le \frac{2t}{n\alpha^2}$. Meanwhile $\frac{1}{c(t)}\le \frac{2(\alpha-1)}{n\alpha^2 K}$, hence we can choose $c(t)=\frac{n\alpha^2}{2}\max\left\{\frac{1}{t}, \frac{K}{\alpha-1}\right\}$ to versify  \eqref{case1}. Although the function  $c(t)=\frac{n\alpha^2}{2}\max\left\{\frac{1}{t}, \frac{K}{\alpha-1}\right\}$ is not $C^1$ in time $t$, we can find some smooth enough functions $c_{\varepsilon}(t)$  satisfying   \eqref{case1} and $c_{\varepsilon}(t)\to c(t)$ as $\varepsilon\to0$, thus we complete the proof.
\end{proof}
\begin{remarks}\label{rem2}
\item{(1).} Combining with Remark \ref{rem1}, we can strengthen the above result to
\begin{equation}\label{cor1-eq2}
 |\nabla \log u|^2-\alpha (\log u)_t\le \begin{cases}\frac{n\alpha^2}{2t}, &t\le \frac{\alpha-1}{K}\\
 \frac{Kn\alpha^2}{4(\alpha-1)}\left(1+e^{-2\left(\frac{K}{\alpha-1}t-1\right)}\right), &t>\frac{\alpha-1}{K}.
\end{cases}\end{equation}

\item{(2).} \eqref{cor1-eq1} and \eqref{cor1-eq2}  improve the Li-Yau inequality \eqref{LY-grad} and Davies's result \eqref{LY-Hamil} respectively.  Also  \eqref{cor1-eq2}  improves  Corollary 1.2, Corollary 1.4  in \cite{Yu-Zhao19} .
\end{remarks}

\begin{proof}[Proof of \eqref{cor1-eq2}]
By Corollary \ref{cor1}, we can choose $c(t)=\frac{n\alpha^2}{2t}$, for $0<t\le t_0:=\frac{\alpha-1}{K}$. For $t>t_0$, we only need  $c(t)$ satisfies  \eqref{case2} with $c(t_0)=\frac{nK\alpha^2}{2(\alpha-1)}$ and $\alpha>1$ is a constant. Solving it with equality gives
$$
c(t)=\frac{Kn\alpha^2}{4(\alpha-1)}\left(1+e^{-2\left(\frac{K}{\alpha-1}t-1\right)}\right), \ t\ge t_0=\frac{\alpha-1}{K}.
$$
It's easy to see  the right hand side in \eqref{cor1-eq2} is $C^1$ in time $t$.  Hence the desired result immediately follows by Remark \ref{rem1}.
\end{proof}

\begin{corollary}\label{cor2}
Let $M$ be a compact manifold without boundary, $Ricci(M)\ge-K (K\ge0)$.
Let $u$ be any positive solution of the  heat equation $u_t=\Delta u $ on $M$. We have for all $t>0$ and $\theta\ge 0$
\begin{equation}\label{cor2-eq1}
|\nabla \log u|^2-e^{2\theta Kt}(\log u)_t\le \frac{ nKe^{4\theta Kt}}{e^{2\theta Kt}-1}\max\left\{\frac{1-\theta}{2}, \theta\right\},
\end{equation}
where $\max\left\{\frac{1-\theta}{2}, \theta\right\}=\theta$ if $\theta\ge\frac{1}3$; $=\frac{1-\theta}{2}$ if $\theta\in [0,\frac13]$.
\end{corollary}
\begin{proof}
Let's take $\alpha=e^{2K\theta t}$. It follows from \eqref{case1}
\begin{equation*}
\begin{cases}
c(t)&\ge \frac{ nKe^{4\theta Kt}}{e^{2\theta Kt}-1}\cdot\frac{1-\theta}{2} \\
c(t)&\ge \frac{n\alpha(t)}{2\int_0^t \frac{1}{\alpha(s)}}ds= \frac{ \theta nKe^{4\theta Kt}}{e^{2\theta Kt}-1} .
\end{cases}\end{equation*}
Hence we can take
$$
c(t)=\frac{ nKe^{4\theta Kt}}{e^{2\theta Kt}-1}\max\left\{\frac{1-\theta}{2}, \theta\right\}.
$$ We can verify  $c(t)$ satisfies \eqref{case1} and $c(0+)=+\infty$. By Theorem \ref{thm1}, we complete the proof.
\end{proof}
If we take $\theta=1$ in Corollary \ref{cor2},  we re-find the refined Hamilton type gradient estimate (see Theorem 2.1 and Theorem 3.1 in \cite{Qian19})
\begin{equation*}\label{Hamil-refine}
|\nabla \log u|^2-e^{2Kt}(\log u)_t\le \frac{nKe^{4Kt}}{e^{2Kt}-1}.
\end{equation*}
In \cite{Qian19}, the local estimate is also obtained, hence \eqref{Hamil-refine} holds in the setting of complete Riemannian manifold.  The corresponding Harnack inequality and heat kernel estimate are  derived.

%%\begin{corollary}\label{cor3}
%%Let $M$ be a compact manifold without boundary, $Ricci(M)\ge-K (K\ge0)$. Let $u$ be any positive solution of the  heat equation $u_t=\Delta u $ on $M$. For $h\in C^1[0,\infty)$ satisfying $h\ge h(0):=1$ and $-K(6\theta-2)\le h'(t)\le 2K(1-\theta)e^{-2K\theta t}$ for $t\ge 0$, where $\theta\in [\frac13,1]$ is a constant,  we have for all $t>0$,
%%\begin{equation}\label{cor3-eq1}
%%%|\nabla \log u|^2-\alpha (\log u)_t\le  \frac{n\alpha}{2\int_0^t \frac{1}{\alpha}ds},
%%\end{equation}
%%where
%%$\alpha(t)=e^{2K\theta t}h(t)$.
%%\end{corollary}
%%\begin{proof}
%%%%Direct computation we have
%%%%$$
%%%%\int_0^t \frac{1}{\alpha}ds\le \frac{1-e^{-2K\theta t}}{2K\theta }\le \frac{2\left(1-e^{-2K\theta t}\right)}{2K(1-\theta)-h'} \le \frac{2(\alpha-1)}{(2K\alpha-\alpha')}.
%%%%$$
%%%%The desired result follows from Theorem \ref{thm-case1}.
%%%%%%%%\end{proof}

In the positive curvature case, we have
\begin{corollary}\label{cor2-2}
Let $M$ be a compact manifold without boundary, $Ricci(M)\ge-K (K<0)$.
Let $u$ be any positive solution of the  heat equation $u_t=\Delta u $ on $M$. We have for all $t>0$ and $\theta\in (0,\frac13]$
\begin{equation}\label{cor2-eq2}
|\nabla \log u|^2-e^{2\theta Kt}(\log u)_t\le \frac{ nK\theta e^{4\theta Kt}}{e^{2\theta Kt}-1}.
\end{equation}

\end{corollary}
\begin{proof}
For $\theta\in (0,\frac13]$, we take $\alpha=e^{2K\theta t}$. It follows from \eqref{case1}
\begin{equation*}
\begin{cases}
c(t)&\le \frac{ nKe^{4\theta Kt}}{e^{2\theta Kt}-1}\cdot\frac{1-\theta}{2} \\
c(t)&\ge \frac{n\alpha(t)}{2\int_0^t \frac{1}{\alpha(s)}ds}= \frac{ \theta nKe^{4\theta Kt}}{e^{2\theta Kt}-1} .
\end{cases}\end{equation*}
Hence we can take
$$
c(t)=\frac{ nK\theta e^{4\theta Kt}}{e^{2\theta Kt}-1}.
$$
We can verify  $c(t)$ satisfies \eqref{case1} and $\lim\limits_{t\downarrow0}c(t)=+\infty$. By Theorem \ref{thm1}, we complete the proof.
\end{proof}

\begin{remark}
 If we take $\theta=1/3$ in Corollary \ref{cor2-2}, we re-find  Theorem 2.1 in \cite{BV12}, they also get the estimate of the associated heat kernel and the lower bounds for the eigenvalues of Laplacian. We can derive the following new Harnack inequality from \eqref{cor2-eq2}: for $0\le s<t$ and $x,y\in M$,
    $$
    u(s,x)\le \left(\frac{1-e^{2\theta Kt}}{1-e^{2\theta Ks}}\right)^{\frac{n}{2}}e^{-\frac{\theta K}{2} \frac{d^2(x,y)}{e^{-2\theta Kt}-e^{-2K \theta s}}}u(t,y), \forall\  0<\theta\le \frac13.
    $$
\end{remark}

%\begin{remarks}

%%\item{(2).} In the case of  $K\ge 0$ (negative curvature) and $\alpha<1$ is a constant, the nonlinear condition holds if
%%\begin{eqnarray}
%%\hskip 50pt c(0+)=\lim_{t\to 0} c(t)& =&+\infty\label{case2-1}\\
%%\hskip 50pt 2K+\frac{4(1-\alpha)}{n\alpha^2}c(t)&>&0 \label{case2-2}\\
%%\left(2K+\frac{4(1-\alpha)c(t)}{n\alpha^2}
%%\right)^2-\frac{8(1-\alpha)^2}{n\alpha^2} \left(\frac{2c^2(t)}{n\alpha^2}+c'(t)\right)&\le&0.\label{case2-3}
%%\end{eqnarray}
%From \eqref{case2-3} we can obtain $$
%%c(t)\le -\frac{n\alpha^2k}{4(1-\alpha)}.
%%$$
%%{\bf Contradiction to (\ref{case2-1})
%%Even if we assume $\alpha>1$,  from \eqref{case2-2} we can obtain $c(t)\le \frac{kn\alpha^2}{2(\alpha-1)}$, still contradiction to \eqref{case2-1}. }
%%
%%\item{(3).} In the case of  $K\le 0$ (positive curvature) and $\alpha<1$
%%
%%
%%\item{(4).} In the case of  $K\ge0$ (negative curvature) and  $\alpha(t)=e^{2Kt}$, we have  the conditions (A3), (B3) follow immediately if
%%$$
%%-\frac{2c^2(t)}{n\alpha^2}-c'(t)+\frac{\alpha'(t)}{\alpha(t)}c(t)\le0 \ \mbox{and }\ c(0)=\lim_{t\to0}c(t)=\infty.
%%$$
%%From which we can get $c(t)\ge \frac{nKe^{4Kt}}{e^{2Kt}-1}$.
%%\end{remarks}

\subsection{Case 2 in Theorem \ref{thm1}}
Take $\alpha$ to be certain expression, we find the following
\begin{theorem}\label{thm-case2}
Let $M$ be a compact manifold without
boundary, $Ricci(M)\ge-K (K\ge0)$.  Let $u$ be any positive solution of the  heat equation $u_t=\Delta u
$ on $M$.For a given $C^1$ positive function
 $a(t):(0,\infty)\to (0,\infty)$, we always suppose $a(t)$  satisfies the following assumptions:\\
 (A1). For all $t>0$, $a(t)>0$ and $\lim_{t\to0}a(t)=0$, $\lim\limits_{t\to0}\frac{a(t)}{a'(t)}=0$. \\
 (A2). For any $L>0$,
 $\frac{a'^2}{a}$ is  continuous and integrable on the interval $[0,L]$.

Then we have for all $t>0$,  \begin{equation}\label{eq-case2} \left|\nabla
\log u\right|^2-\alpha(t)(\log u)_t\le c(t), \end{equation}
where
 \begin{equation}\label{eq-case2-1}
\alpha(t)=\frac{2K}{a(t)}\int_0^ta(s)ds+1, \
c(t)=\frac{nK}{2}+\frac{nK^2}{2a(t)}\int_0^ta(s)ds+\frac{n}{8a(t)}\int_0^t\frac{a'^2(s)}{a(s)}ds.
 \end{equation}
\end{theorem}
\begin{proof}
We only need to verify Assumption (a) and  {\bf Case 2}  of (b) in Theorem \ref{thm1}, i.e. \eqref{case2}.  For such function $a$, we take $\alpha(t)=\frac{2K}{a(t)}\int_0^ta(s)ds+1$, direct computation gives: $c(t)$ defined in \eqref{eq-case2-1} satisfies \eqref{case2} and $c(0+)=+\infty$.
\end{proof}

\begin{remarks}%\label{cor3}
\item{(1).} If we take $a(t)=t^2$ and $a(t)=\sinh^2(Kt)$, we can get the estimates \eqref{LX-eq1} and \eqref{LX-eq2} respectively. Different choice of $a(t)$ gives various differential Harnack inequalities, e.g. see \cite{Qian14}.
\item{(2).} Compared with Theorem 1.1 in \cite{Qian14}, we drop the assumption of $a'>0$. In fact, going through carefully the proof of Theorem 1.1 in \cite{Qian14}, we can also drop the assumption of $a'>0$.
\item{(3).} We mention here that: For Theorem 1.1 and Theorem 1.2 in \cite{Li-Xu11}, it {\bf does} need the assumption of nonpositive Ricci curvature ($K\ge0$) for the local differential Harnack inequality, since in their proof it is necessary to assume $\alpha>1$, see Line-4 of Page 4468. For the compact manifolds with convex boundary, the global differential Harnack inequalities \eqref{eq-case2} works both for the negative curvature and positive curvature for time $t$ satisfying $\alpha(t)>0$, since the proof of Theorem 1.1 in \cite{Qian14} works in this case. Indeed  (2.8) in \cite{Qian14} gives   $(\Delta-\partial_t)(aF)\ge -2\nabla (aF)\cdot\nabla f$ and we consider $aF$ instead of $F$ after (2.8).
\end{remarks}

Combining Corollary \ref{cor1}, Remarks \ref{rem2}, Corollary \ref{cor2},Corollary \ref{cor2-2} and Theorem \ref{thm-case2}, we see that Theorem \ref{thm1} or more generally Theorem \ref{YHM-thm0} provides a general form of  differential Harnack inequality, which unifies and partial improves the classical Li-Yau inequality \eqref{LY-grad}, Davies's estimate \eqref{LY-Hamil},  Hamilton differential Harnack inequality \eqref{Hamil}, Li-Xu's result \eqref{LX-eq1},  \eqref{LX-eq2}, Baudoin-Garofalo \cite{BG11}, Qian \cite{Qian14} etc. %%Furthermore we explain: the classical Li-Yau inequality \eqref{LY-grad}, Hamilton differential Harnack inequality \eqref{Hamil} is corresponding to the Case 1 in Proposition \ref{prop-basic}

%{\bf Acknowledgement}:   The author is  greatly indebted to Prof. L. M. Wu (Clermont-Auvergne, France), Prof. D.
%Bakry (Toulouse University III, France),  Prof. X. D. Li (Institute of Applied Mathematics, Academia Sinica, China, China) and Prof. H. Q. Li (Fudan university) for their constant %encouragement and support.  Sponsored by the NSF of China (No. 11671076) and Qing Lan Project of Jiangsu.


\begin{thebibliography}{}

\bibitem{BBG17}
D.~Bakry, F.~Bolley, and I.~Gentil.
\newblock The {L}i-{Y}au inequality and applications under a
  curvature-dimension condition.
\newblock {\em Ann. Inst. Fourier}, 67(1):397--421, 2017.

\bibitem{Bakry-Ledoux06} D. Bakry, M. Ledoux, A logarithmic Sobolev
 form of the Li-Yau parabolic inequality,  Revista Mat.
 Iberoamericana, 22 (2006) 683-702.

\bibitem{BQ99} D. Bakry,  Z. Qian, Harnack inequalities on a manifold with positive or negative Ricci curvature, Rev. Mat. Iberoam. 15 (1999), no. 1,  143-179.


\bibitem{Davies89} E. B. Davies, Heat kernels and spectral theory, Cambridge Tracts in Mathematics, vol. 92, Cambridge University Press, Cambridge, 1989, x+197 pages.


\bibitem{BG11} F. Baudoin, N. Garofalo, Perelman's entropy and doubling property on Riemannian manifolds J. Geom. Anal. 21 (2011), no. 4, 1119-1131.

\bibitem{BV12} F. Baudoin, A. Vatamanelu, A note on lower bounds estimates for
the Neumann eigenvalues of  manifolds with positive Ricci curvature, Potential Analysis, 37(1) (2012), 91-101.

\bibitem{Hamilton93} R. S. Hamiltom, A Matrix harnack estimate
for the heat equation, Comm. in Ana. and Geom. 1 (1993) 113-126.


\bibitem{Li-Xu11} J. Li, X. Xu, Differential Harnack inequalities on Riemannian manifolds I: linear heat equation Adv. Math. 226 (2011), no. 5, p. 4456-4491.


\bibitem{LY} P. Li, S. T. Yau,  On the parabolic kernel of the Schr\"odinger operator,
  Acta. Math. 156 (1986), 153-201.

\bibitem{MZS08}L. Ma, L. Zhao, X. F. Song,  Gradient estimate for the degenerate
parabolic equation $u_t=\Delta F(u)+H(u)$ on manifolds, J. Diff.
Equ. 224 (2008) 1157-1177.


\bibitem{Qian14} B. Qian, Remarks on differential Harnack inequalities, J. Math. Anal. Appl. 409 (2014), no. 1, 556-566.

\bibitem{Qian19} B. Qian, Refined Hamilton type differential Harnack inequality,  preprint, 2019.


\bibitem{SZ06}P. Souplet, Q. S. Zhang,  Sharp gradient estimate and Yau's liouvill theorem for the heat equation on noncompact manifolds, Bull. London Math. Soc. 38(2006) 1045-1053.


\bibitem{Yau94} S. T. Yau, On the Harnack inequalities for partial
differential equations, Comm. Anal. Geom., 2 (1994), 431-450.

\bibitem{Yau95} S. T. Yau, Harnack inequality for non-self-adjoint
evolution equations. Math. Research Lett., 2 (1995), 387-399.

\bibitem{Yu-Zhao19} C. J. Yu, F. F. Zhao, A note on Li-Yau type gradient estimate, Acta Math. Sci. (B), 39 (2019), no. 4, 1185-1194.

\bibitem{Yu-Zhao20} C. J. Yu, F. F. Zhao, Sharp Li-Yau-type gradient estimates on hyperbolic spaces, J. Geom. Anal. 30 (2020), no. 1, 54-68.




 \end{thebibliography}
\end{document}